\documentclass[reqno, 12pt]{amsart}
\usepackage{amsfonts, amsmath, amssymb, amscd, amsthm}
\usepackage{mathrsfs}  
\usepackage{bbm}
\usepackage{color}
\usepackage{rotating}
\usepackage{graphics}
\usepackage{xcolor}

\usepackage
[
a4paper,
vmargin=3cm,
hmargin=4cm
]{geometry}
\RequirePackage[bookmarks,%
colorlinks,%
urlcolor=prune,%
citecolor=red,%
linkcolor=blue,%
hyperfigures,%
pdfcreator=LaTeX,%
breaklinks=true,%
pdfpagelayout=SinglePage,%
bookmarksopen=true,%
bookmarksopenlevel=2]{hyperref}
\usepackage{float}
\usepackage{caption} 
\usepackage{subcaption}

\theoremstyle{plain}
\numberwithin{equation}{section}
\newtheorem{theorem}{Theorem}[section]

\newtheorem{lemma}[theorem]{Lemma}

\theoremstyle{definition}
\newtheorem{remark}[theorem]{Remark}

\def \div {\mathrm{div}}
\def \e {\epsilon}
\def \p {\partial}
\def \Je {\mathcal{J}_{\epsilon}}

\newcommand{\norm}[1]{\left\lVert#1\right\rVert}    
\newcommand\abs[1]{\left|#1\right|}    

\begin{document}
\title{On the blow up of a non-local transport equation in compact manifolds}

\author{Diego Alonso-Or\'{a}n}
\address{Departamento de Analisis Matem\'atico, Universidad de La Laguna, C/ Astrof\'isico Francisco Sanchez s/n, 38271,  Spain}
\email{dalonsoo@ull.edu.es}
\author{\'Angel David Mart\'inez}
\address{Fields Ontario Postdoctoral Fellow\\ University of Toronto Mississauga, Toronto, Ontario, Canada}
\email{martinez@math.toronto.edu}


\date{\today}

\begin{abstract}
In this note we show finite time blow-up for a class of non-local active scalar equations on compact Riemannian manifolds. The strategy we follow was introduced by Silvestre and Vicol to deal with the one dimensional C\'ordoba-C\'ordoba-Fontelos equation and might be regarded as an instance of De Giorgi's method. 
\end{abstract}
\maketitle
\section{Introduction}
The problem of global regularity versus finite time blow up for active scalar equation with non-local velocities has received a lot of attention in recent years. In their seminal paper  C\'ordoba, C\'ordoba and Fontelos introduced the following non-local equation
\begin{equation}\label{CCF}
 \partial_t\theta-\mathcal{H}\theta\cdot\theta_x=0
\end{equation}
where $\mathcal{H}$ is the Hilbert transform (cf. \cite{CCF}). The equation constitutes a one dimensional model for the two-dimensional surface quasi-geostrophic equation (SQG), as well as for the Birkhoff-Rott equations describing the evolution of vortex sheets with surface tension \cite{Bakeretal, Morlet}. In the aforementioned paper \cite{CCF} the authors first proved the finite time blow-up of classical solutions to \eqref{CCF} for a generic class of smooth initial data. Their delicate argument relies on an integral inequality shown by means of the Mellin transform and complex analysis.  Subsequent works have  shown finite blow-up avoiding complex analysis techniques (cf. \cite{Kiselev1}). In particular, in \cite{SV} Silvestre and Vicol provided four elegant and simple proofs of the blow-up phenomena. \\ 

Another model including dissipative effects to the equation \eqref{CCF} is
 
\begin{equation}\label{CCFd}
\partial_t\theta-\mathcal{H}\theta\cdot\theta_x +\kappa \Lambda^{\gamma}\theta=0 ,
\end{equation}
where $\kappa \geq 0$, $\gamma\in (0, 2]$ are fixed parameters and $\Lambda^{\gamma}=(-\Delta)^{\gamma/2}$ is the fractional Laplacian. In \cite{CCF}, the authors also obtained global well posedness for positive $H^{2}$ data in the subcritical case $\gamma>1$ and for small initial data in the critical case $\gamma=1$. Later, the global regularity in the critical case was shown in \cite{Dong1} for arbitrary initial data in Sobolev spaces by adapting the method of continuity in \cite{KNV}. In \cite{LiRodrigo1}, finite time singularities of smooth solutions was shown in the range $0<\gamma<\frac{1}{2}$. For the intermediate case, $\frac{1}{2}\leq \gamma < 1$,  whether solutions may blow up in finite time is an intriguing open problem. \\

Higher dimensional analogues  of equation \eqref{CCFd} have also been studied extensively. In \cite{LiRodrigo2}, the authors studied the following non-local equation given by
\begin{equation}
\left\{
\begin{array}{lcl}\label{GCCFd}
\partial_{t}\theta+u\cdot \nabla \theta +\kappa \Lambda^{\gamma}\theta=0, \quad \mbox{on } \mathbb{R}^{2} \times(0,\infty)\\
\theta(x,0)=\theta_{0}(x), \quad \mbox{in }\mathbb{R}^{2}
\end{array}
\right.
\end{equation}
where the velocity $u$ is defined by
\begin{equation*}
u=\mathcal{Q}_{\beta}\Lambda^{-1}\nabla\theta
\end{equation*}
with
$$ \mathcal{Q}_{\beta}=\begin{pmatrix}
\cos{\beta} & -\sin{\beta} \\
\sin{\beta} & \cos{\beta}
\end{pmatrix}, \quad \beta\in [0,2\pi).$$
If $\beta=\pi/2,3\pi/2$, we recover the SQG equation (where the velocity field is incompressible). This equation was derived to model frontogenesis in meteorology, a formation of sharp fronts between masses of hot and cold air. It is an important example of 2D active scalars and in the context of geophysical fluid dynamics, the variable $\theta$ denotes the temperature (or surface buoyancy
function) in a rapidly rotating stratified fluid with uniform potential vorticity \cite{HPGS95,Ped}. We refer the reader to \cite{CMT, Ju, Kiselev1, KNV,R}  for more details both from the theoretical and numerical point of view. For $\beta \in [0,2\pi)\setminus \{ \pi/2,3\pi/2\}$ and $\kappa=0$, Dong and Li showed in \cite{DongLi1} the blow up of smooth radial solutions, while Balodis and C\'ordoba  \cite{BC} proved that finite time blow up occurs when $\beta = 0$ or $\pi$ without the radial restriction. Their proof, based on an integral inequality for the Riesz transform, applies also for higher dimensions $n\geq 2$. Later, under similar assumptions, finite time singularity formation was obtained in \cite{LiRodrigo2} in the supercritical range $0<\gamma<\frac{1}{2}$.\\

This short note originated while the authors studied the possibility of finite time blow-up of smooth solutions for the following nonlinear and non-local active scalar equation on compact Riemannian manifolds $(M,g)$  
\begin{equation}
\left\{
\begin{array}{lcl}\label{eq:active:scalar}
\partial_{t}\theta+u\cdot\nabla_{g} \theta =0, \quad \mbox{on }M\times(0,\infty)\\
\theta(x,0)=\theta_{0}(x), \quad \mbox{in }M
\end{array}
\right.
\end{equation}
where $\theta(x,t)$ is a transported scalar-valued function, $u$ is the velocity field determined by the constitutive relation
\begin{equation}\label{constituve:law2}
u=\nabla_{g}\Lambda_{g}^{-1}\theta
\end{equation}
and $\Lambda_{g}=(-\Delta_{g})^{\frac{1}{2}}$ denotes the fractional Laplace-Beltrami operator. The authors interest in this type of blow up scenario, in the particular case $M=\mathbb{S}^2$, came naturally by analogy with the global existence of classical solutions to the critical SQG equation on the sphere  (cf. \cite{AOCM, AOCM2}). At the time, the authors unsuccesfully tried to extend the work of Balodis and C\'ordoba on the relevant non-local nonlinear inequality using a (spherical) harmonic analysis approach. Quite surprisingly the ideas of \cite{SV} can be adapted to general Riemannian manifolds. The proof works by contradiction, using De Giorgi's iteration scheme to derive a decay for the $L^\infty$ norm of the active scalar. This decay implies that the solution must develop a singularity in finite time. 

We can now state a first result in this direction, namely:
\begin{theorem}\label{Th:1}
There are smooth initial conditions $\theta_{0}(x)\in L^1(M)\cap L^{\infty}(M)$ for which \eqref{eq:active:scalar}-\eqref{constituve:law2} blows up in finite time.
\end{theorem}

It should be noted that even in the two dimensional torus Theorem \ref{Th:1} provides a new result that seems to be out of reach \textcolor{black}{using the classical C\'ordoba, C\'ordoba, Fontelos inequalities first proved in \cite{CCF} nor can one argue radially as in the full euclidean space \cite{DongLi1, LiRodrigo2}}. Furthermore, any perturbation of the metric might affect the subtle non-local integral inequalities involved. This leaves open the question whether the results regarding the finite time blow-up of the dissipative equation analogous to \eqref{CCFd} can be extended to this setting (even in the range  $0<\gamma<\frac{1}{2}$).\\

Before stating the main result, which is stronger than Theorem \ref{Th:1}, let us digress and introduce the so-called interpolating models between the two dimensional Euler equation and the singular surface quasi-geostrophic equation and their dissipative analogues. These had already been extensively studied in the literature, namely
\begin{equation}
\left\{
\begin{array}{lcl}\label{eq:alpha}
\partial_{t}\theta+u\cdot \nabla \theta +\kappa \Lambda^{\gamma}\theta=0, \quad \mbox{on } \mathbb{R}^{2} \times(0,\infty)\\
\theta(x,0)=\theta_{0}(x), \quad \mbox{in }\mathbb{R}^{2}
\end{array}
\right.
\end{equation}
where the velocity $u$ reads
\begin{equation}\label{constituve:law3}
u =  \Lambda^{-1+ \alpha}\nabla \theta
\end{equation}
for $-1\leq\alpha\leq 1$ (cf. \cite{DongLi1,DongLi2,Chae,CY,LiRodrigo3}). We refer the reader to the aforementioned papers and their references therein for a more extended analysis in this area.\\

This suggest considering the following interpolating nonlinear and non-local active scalar equation on compact Riemannian manifolds $(M,g)$  
\begin{equation}
\left\{
\begin{array}{lcl}\label{eq:active:scalar2}
\partial_{t}\theta+u\cdot\nabla_{g} \theta =0, \quad \mbox{on }M\times(0,\infty)\\
\theta(x,0)=\theta_{0}(x), \quad \mbox{in }M
\end{array}
\right.
\end{equation}
where $\theta(x,t)$ is a scalar-valued function and $u$ is the velocity field determined by the constitutive relation
\begin{equation}\label{constituve:law}
u=\nabla_{g}\Lambda_{g}^{-1+\alpha}\theta,  \textcolor{black}{\mbox{ with } \alpha\in(-1,1).}
\end{equation}
The range of $\alpha$ interpolates in one dimension between the one-dimensional version of the 2D Euler vorticity equation, corresponding $\alpha=-1$, and the Hamilton-Jacobi equation, for $\alpha=1$.\\

Finally, let us state the main result:

\begin{theorem}\label{Th:2}
There is no classical global solution to \textcolor{black}{\eqref{eq:active:scalar2}-\eqref{constituve:law}} for any positive smooth integrable initial datum with $\|\theta_0\|_{\infty}=1$ and small enough norm $\theta_{0}(x)$ in $L^1(M)$ .
\end{theorem}

In fact, the proof will provide that for {\em any} positive smooth integrable datum $\theta_0$ with small enough norm (depending on $\|\theta_0\|_{\infty}$) there is no classical global solution.\\

 Notice that Theorem \ref{Th:1} is a consequence of this result for $\alpha=0$. As in \cite{SV}, we implicitly use a Virial type identity yielding a dissipative term  which can be used to implement De Giorgi's technique. In our setting though the relevant interpolation inequality contains an extra term and is not scale invariant (cf. Lemma \ref{inter:ine:manifold}). Moreover, as the metric is fixed we are not allowed to use scaling arguments. To go around this difficulty we keep track of the time dependence in the Giorgi's nonlinear inequality. This is ultimately responsible for the restrictions in the statement. Let us notice that the original argument applies verbatim to the euclidean space $\mathbb{R}^n$. \footnote{ \textcolor{black}{During the peer review process of this article, a new preprint \cite{Jiu-Zhang} showed (as indicated above) that indeed Theorem \ref{Th:2} also holds in the euclidean space $\mathbb{R}^n$.}}\\

The manuscript is organized as follows. In Section \ref{section2} we present some basic preliminary observations and show an interpolation inequality on compact manifolds. We also fix the notations we will employ.
Section \ref{section3} contains the proof of the main result. For the sake of completeness, in Appendix \ref{appendix}, we provide the proof of the local well-posedness of the system \eqref{eq:active:scalar}-\eqref{constituve:law2} in the Sobolev spaces defined below. 

\section{Preliminary observations and interpolating inequality}\label{section2}
Let us introduce some notation first, if we denote the metric on $M$ by \[g=\sum_{j,k}g_{jk}(x)dx_jdx_k\]
recall that the Laplace-Beltrami operator associated with it is given by
\[\Delta_g=\frac{1}{\sqrt{|g|}}\sum_{j,k}\frac{\partial}{\partial x_j}\left(\sqrt{|g|}g^{jk}\frac{\partial}{\partial x_k}\right)\]
where $(g^{jk})=(g_{jk})^{-1}$ and $dx$ denotes the associated volume form as usual. Then the eigenvalues of $-\Delta_g$ are non-negative, numerable and one can find a basis given by the corresponding eigen-functions $\{\phi_k\}$. The fractional powers $\Lambda_g^{\alpha}$ of $-\Delta_g$, $0\leq\alpha\leq 2$ can be described spectrally as the linear operator that satisfies $\Lambda_g^{\alpha}\phi_k=\lambda_k^{\alpha}\phi_k$ for any $k$. We will also denote by $\nabla_{g}:C^{\infty}\to \Gamma_{C^\infty}(TM)$ the gradient operator where $TM$ is the tangent bundle and the divergence operator by $\textrm{div}_{g}:\Gamma_{C^\infty}(TM)\to C^{\infty}(M)$. Let us recall that in local coordinates the expresions for the gradient and the divergence operator are
\[\nabla_gf=\sum_{j,k}g^{jk}\frac{\partial f}{\partial x_k}\partial_{x_j}\textrm{ and } \div_gX=\frac{1}{\sqrt{|g|}}\sum_{j,k}\frac{\partial}{\partial x_j}\left(\sqrt{|g|}X^j\right),\]
respectively. It is clear that $\Delta_{g}=\textrm{div}_{g}\circ\nabla_{g}$. We want to emphasize that this are local operators.\\

For $1\leq p<\infty$ we denote by $L^p(M)$ the standard Lebesgue space and by $L^\infty(M)$ the space of essentially bounded functions. We also denote by $\langle f, g\rangle$ the inner $L^2$ product of $f,g$ on the manifold. The notation $a\lesssim b$ means there exists $C$ such that $a \leq Cb$, where $C$ is a positive universal constant that may depend on fixed parameters, constant quantities, and the manifold itself. Note also that this constant might differ from line to line.
\\

Recall that it is well known that the Laplace-Beltrami operator on compact manifolds has a discrete set $\phi_k$ of normalized eigen-functions, i.e. $-\Delta_{g} \phi_k=\lambda_k^2 \phi_k$ with $\|\phi_k\|_{L^2}=1$ where the eigenvalues $0=\lambda_0\leq\lambda_1\leq\lambda_2\leq\cdots$ tend to infinity. We  define the homogeneous Sobolev spaces $\dot{H}^{s}(M)$ for $s\in\mathbb{R}$ spectrally by
\begin{equation}\label{Spectral:Sobolev}
  \dot{H}^{s}(M)=\{ f\in L^2(M): \displaystyle\sum_{k=0}^{\infty} \lambda^{2s}_{k} a_{k}^{2}<\infty \}  
\end{equation}  
where $a_{k}=\left<f, \phi_k \right>=\int_{M} f(x) \phi_k(x) \  d\textrm{vol}_g(x)$.

The full space $H^s$ would correspond to the same adding the zero term $\|f\|_{L^2}$ which is killed as the eigenvalue of constant eigen-functions vanishes.\\

With this notation at hand let us briefly recall some well-known observations regarding the transport active scalar equations \eqref{eq:active:scalar}. We first state the $L^{\infty}$ maximum principle and the pointwise inequality for the fractional Laplace-Beltrami operator
\begin{lemma}\label{maximum:principle}(Maximum principle)
Let $\theta$ be a classical solution of \eqref{eq:active:scalar}. Then the supremum and the infimum of $\theta$ are not increasing in time and the following estimate holds
\begin{equation}
    \norm{\theta(\cdot,t)}_{L^\infty}=\norm{\theta_{0}}_{L^\infty}, \mbox{ for } \quad t>0.
\end{equation}
\end{lemma}

The equation is a transport equation for a continuous velocity provided $\theta$ is smooth enough. This shows that the flow-map diffeomorphism it induces will not be able to reduce the maximum of $\theta$. Another useful observation is
\begin{lemma}(C\'ordoba-C\'ordoba pointwise inequality)
Given $\alpha\in (0,2]$ and a convex function $\phi\in C^1(\mathbb{R})$ the following pointwise inequality holds
\begin{equation}\label{pointiwse:ine}
    \Lambda_{g}^{\alpha}(\phi(f))(x) \leq \phi'(f(x))\cdot \Lambda_{g}^{\alpha}f(x)
\end{equation}
for any $f\in C^{\infty}(M)$.
\end{lemma}
It appeared originally in \cite{CC} in the flat context. In that case a computation using Fourier analysis and an ingenious identity on the integrands provide a proof. In the general context of compact manifolds this was proved later in \cite{CM} where we refer for a complete proof.

The following lemma deals with an interpolating inequality on compact manifolds

\begin{lemma}
Let $(M,g)$ be a compact manifold and $f:M\to \mathbb{R}$ be a smooth function. Then the following inequality holds true
\begin{equation}\label{inter:ine:manifold}
\|f\|_{L^2}\lesssim \|f\|_{L^1}^{\frac{1+\alpha}{2n+\alpha}}\|f\|_{\dot{H}^{\frac{1+\alpha}{2}}}^{\frac{2n-1}{2n+\alpha}}+\|f\|_{L^1}
\end{equation}
for any $\alpha\in(-1,1)$.
\end{lemma}

\begin{proof}
We will use the Sobolev spaces defined in \eqref{Spectral:Sobolev}. On the one hand we have that
\[\norm{f}^2_{L^2}=\sum_{k=0}^{\infty}|a_k|^2=\sum_{\lambda_k\leq R}|a_k|^2+\sum_{\lambda_k >R}|a_k|^2.\]
One can estimate the first term using Weyl's law
$$ N(R)=\abs{\{k:\lambda_{k}\leq R\}}= \frac{\omega_{n}\textrm{vol}_{g}(M)}{(2\pi)^{n}}R^{n}+O(R^{n-1}), $$
cf. \cite{Ch,So}, we have that 
$$ \sum_{\lambda_k\leq R}|a_k|^2 \leq C(n,M) R^n\sup_{\lambda_k\leq R}|a_k|^2 $$
and hence 
\[\norm{f}^2_{L^2}\lesssim  |a_0|^2+R^n\sup_{\lambda_k\leq R}|a_k|^2+\frac{1}{R^{1+\alpha}}\sum_{\lambda_k>R}|a_k|^2\lambda_k^{1+\alpha}.\]
Invoking H\"ormander's bound 
\[\norm{\phi_k}_{L^\infty}\leq C(M,g)\lambda_k^{\frac{n-1}{2}},\]
cf. \cite{Ho}, and using the Sobolev spaces spectral definition \eqref{Spectral:Sobolev} we infer 
\begin{align*}
    \norm{f}^2_{L^2} &\lesssim |a_0|^2+R^n\sup_{\lambda_k\leq R}\left(\|f\|_{L^1}\lambda_k^{\frac{n-1}{2}}\right)^2+\frac{1}{R^{1+\alpha}}\norm{f}^{2}_{\dot{H}^{\frac{1+\alpha}{2}}} \\
    &\lesssim \norm{f}_{L^1}^{2}+R^{2n-1}\norm{f}^{2}_{L^1}+\frac{1}{R^{1+\alpha}}\norm{f}^{2}_{\dot{H}^{\frac{1+\alpha}{2}}}
\end{align*}
Picking $R=\left(\dfrac{\norm{f}_{\dot{H}^{\frac{1}{2}}}}{\norm{f}_{L^1}}\right)^{\frac{2}{2n+\alpha}}$ and taking square roots the desired inequality \eqref{inter:ine:manifold} follows.
\end{proof}

\begin{remark}
In the case of the flat torus $M=\mathbb{T}^n$ the proof adapts  yielding
$$ \norm{f}_{L^2}\lesssim \norm{f}_{L^1}^{\frac{1+\alpha}{n+1+\alpha}}\norm{f}_{\dot{H}^{\frac{1+\alpha}{2}}}^{\frac{n}{n+1+\alpha}}+\norm{f}_{L^1} $$
which has the correct scaling. 
\end{remark}

\section{Proof of Theorem \ref{Th:2}}\label{section3}
The proof follows closely the strategy in \cite{SV} based on reminiscent ideas of De Giorgi's iterative scheme.\\

We will assume that $\theta$ is a positive global classical solution to \eqref{eq:active:scalar2} and reach a contradiction. Recall that the first assumption follows easily from the maximum principle \eqref{maximum:principle} since $\theta_{0}>0$. To implement De Giorgi's method, let us define the truncation levels 
\[\ell_{k}=K(1-2^{-k})\]
where $K$ a positive constant that will be chosen later and the truncated functions $\theta_{k}=(\theta-\ell_{k})^{+}=\displaystyle\max\{0,\theta-\ell_{k}\}$. Since $\theta$ solves \eqref{eq:active:scalar2}, the truncated function $\theta_{k}$ solves
\begin{equation}\label{active:scalar:k}
    \p_{t}\theta_{k}+u\cdot \nabla_{g} \theta_{k}=0.
\end{equation}
Integrating \eqref{active:scalar:k} in space and using integration by parts we have that
$$\p_{t}\int_{M}\theta_{k} \ d\textrm{vol}_g(x)=-\int_{M} \Lambda_{g}^{1+\alpha}\theta \theta_{k}\ d\textrm{vol}_g(x).$$
Invoking the C\'ordoba-C\'ordoba pointwise inequality \eqref{pointiwse:ine} for the convex function $\phi_k(x)=\max\{0,x-\ell_k\}$, notice $\phi_k(\theta)=\theta_{k}$, by definition, and the identity $\phi'_{k}(x)=\mathbbm{1}_{\{x>\ell_{k}\}}$ we infer that
\begin{equation*}
    -\Lambda_{g}^{1+\alpha}\theta\theta_{k}=-\phi'_{k}(\theta)\Lambda_{g}^{1+\alpha}\theta\theta_{k}\leq -\Lambda_{g}^{1+\alpha}(\phi_{k}(\theta))\theta_{k}=-\Lambda_{g}^{1+\alpha}\theta_{k}\theta_{k},
\end{equation*}
and hence
\begin{equation}\label{virial:identity:k}
\p_{t}\int_{M}\theta_{k} \ d\textrm{vol}_g(x) \leq-\norm{\theta_{k}}^{2}_{\dot{H}^{\frac{1+\alpha}{2}}}.
\end{equation}
Let $t_{\star}>0$ be fixed (to be specified later in the proof), and consider the increasing sequence of times $t_{k}$ for every $k\geq 0$ converging to $t_{\star}$ as $k\to \infty$, i.e. $0=t_{0}<t_{1}<t_{2}<\cdots < t_{\star}$. For each $k\geq 1$, we will construct a sequence such that $t_{k}\in \left(t_{\star}(1-2^{-k+1}), t_{\star}(1-2^{-k})\right)$ and define the \textit{energy} quantity 
$$E_{k}=\int_{M}\theta_{k}(x,t_{k}) \ d\textrm{vol}_g(x). $$

Let us construct the sequence now.  Dropping the negative term on right hand side in equation \eqref{virial:identity:k} and integrating in the time interval $[t_k,t_{k+1}]$ we have that
\begin{equation}\label{bound:ek}
\int_{M}\theta_{k}(x,t_{k+1}) \ d\textrm{vol}_g(x) \leq \int_{M}\theta_{k}(x,t_{k}) \ d\textrm{vol}_g(x).
\end{equation}
Let us define $T_k=t_{\star}(1-2^{-k})$ which will help us in the construction. Integrating equation \eqref{virial:identity:k} between $T_{k}$ and $t_{\star}(1-2^{-k-1})$ yields
\begin{equation}
\int_{M}\theta_{k}(x,t^{\star}(1-2^{-k-1})) \ d\textrm{vol}_g(x) + \int_{T_k}^{t^{\star}(1-2^{-k-1})} \norm{\theta_{k}}^{2}_{\dot{H}^{\frac{1+\alpha}{2}}} \ dt \leq \int_{M}\theta_{k}(x,T_{k}) \ d\textrm{vol}_g(x).
\end{equation}
An application of the mean value theorem shows the existence of \[t_{k+1}\in(T_k,T_{k+1})=\left(t_{\star}(1-2^{-k}), t_{\star}(1-2^{-k-1})\right)\] such that
\begin{eqnarray}\label{dissipation:bound:k}
\norm{\theta_{k}(x,t_{k+1})}^{2}_{\dot{H}^{\frac{1+\alpha}{2}}} &\leq& \frac{1}{ (t_{\star}(1-2^{-k-1})-T_{k})} \int_{M}\theta_{k}(x,T_{k}) \ d\textrm{vol}_g(x) \nonumber \\
&\leq&\frac{2^{k+1}}{t_{\star}}\int\theta_k(x,t_k) \nonumber \\
&\leq & \frac{2^{k+1}}{t_\star} \int_{M}\theta_{k}(x,t_{k}) \ d\textrm{vol}_g(x)=: \frac{2^{k+1}}{t_\star} E_{k},
\end{eqnarray}
where we have used that, by construction, $t_k<T_k$ and the decreasing behaviour of the $L^1$ norm. This concludes the construction of a sequence $t_k$ converging to the given  $t_{\star}$, as claimed.\\

Let us move and prove the relevant De Giorgi's nonlinear type inequality for the energy $E_k$. Using that the level sets increase one can show
\begin{equation}\label{Straight:1}
    \theta_{k+1}(x,t_{k+1})\leq \theta_{k}(x,t_{k+1})\mathbbm{1}_{\{x:\theta_{k+1}(x,t_{k+1})>0\}}
\end{equation}
and, by a straightforward computation, that
\begin{equation}\label{Straight:2}
    \mathbbm{1}_{\{x:\theta_{k+1}(x,t_{k+1})>0\}}\leq \frac{2^{k+1}}{K}\theta_{k}(x,t_{k+1}).
\end{equation}
Therefore, using \eqref{Straight:1}-\eqref{Straight:2}
\begin{eqnarray*}
 \int_{M}\theta_{k+1}(x,t_{k+1}) \ d\textrm{vol}_g(x) &\leq & \int_{M}\theta_{k}(x,t_{k+1})\mathbbm{1}_{\{x:\theta_{k+1}(x,t_{k+1})>0\}} \ d\textrm{vol}_g(x) \\
    &\leq & \frac{2^{k+1}}{K} \int_{M}\theta^{2}_{k}(x,t_{k+1}) \ d\textrm{vol}_g(x) \\
    &\leq & C\frac{2^{k+1}}{K} \left(\norm{\theta_{k}(\cdot,t_{k+1})}_{L^1}^{\frac{2(1+\alpha)}{2n+\alpha}}\norm{\theta_{k}(\cdot, t_{k+1})}_{\dot{H}^{\frac{1+\alpha}{2}}}^{\frac{2(2n-1)}{2n+\alpha}}+\norm{\theta_{k}(\cdot,t_{k+1})}_{L^1}^2\right)
\end{eqnarray*}
where in the last step we have invoke the interpolation inequality \eqref{inter:ine:manifold}. Estimate \eqref{bound:ek} and \eqref{dissipation:bound:k}  imply that
\begin{equation}\label{final:Ek:recurrence}
E_{k+1}\leq C \frac{2^{(k+1)(1+\gamma)}}{Kt_{\star}^{\gamma}}E_{k}^{\beta}+\frac{2^{k+1}}{K}E_k^2
\end{equation}
with $\beta=\frac{2n+2\alpha+1}{2n+\alpha}>1$ and $\gamma=\frac{2n-1}{2n+\alpha}$. Our goal now will be to prove (by induction) that the sequence $E_{k}$ converges to zero when $k$ tends to infinity.\\

To that end, we first impose a smallness condition on $E_{0}$, namely \[E_{0}^{2-\gamma}\leq t_{\star}^{-\gamma}\]
and choose \[K=\frac{C'}{t_{\star}^{\gamma}}\] with $C'>1$. Let us introduce the induction hypothesis we will use, namely 
\begin{equation}\label{induction:hypothesis}
E_k\leq t_{\star}^{-\gamma/(2-\gamma)}\textrm{ and } E_{k+1}\leq C \frac{2^{(k+1)(1+\gamma)}}{C'}E_{k}^{\beta}.
\end{equation}
For $k=0$, the induction hypothesis is satisfied. Indeed, we have that
$$
E_{0}\leq t_{\star}^\frac{-\gamma}{(2-\gamma)} \textrm{ and }  E_{1}\leq C \frac{2^{1+\gamma}}{C'}E_{0}^{\beta}
$$
where we have used the smallness condition on $E_0$, the fact that $K=\frac{C'}{t_{\star}^{\gamma}}$ and the recurrence relation \eqref{final:Ek:recurrence}.\\

We are now ready to provide the induction step, that is, that \eqref{induction:hypothesis} yields
$$ E_{k+1}\leq t_{\star}^{-\gamma/(2-\gamma)}\textrm{ and } E_{k+2}\leq C \frac{2^{(k+2)(1+\gamma)})}{C'}E_{k+1}^{\beta}.$$
To check the first condition we shall prove an intermediate stronger statement. Namely that if $E_0=\epsilon$ for a small enough $\epsilon=\epsilon(\beta, C, C')$ then the nonlinear inequality in \eqref{induction:hypothesis} implies that
$E_{k+1}\leq \epsilon^{(1-\eta)\beta^{k+1}}$ for some $\epsilon<1$ and small enough $\eta=\eta(\beta)>0$.

Indeed, using the right hand side in \eqref{induction:hypothesis} and denoting by $\tilde{C}=\frac{C}{C'}$ it is easy to infer that
\begin{align*}
E_{k+1}&\leq \tilde{C}2^{(k+1)(1+\gamma)} E_k^{\beta}\\
&\leq \tilde{C}2^{(k+1)(1+\gamma)}\left( \tilde{C}2^{k(1+\gamma)}E_{k-1}^{\beta} \right)^{\beta}\\
&= \tilde{C}^{1+\beta}2^{(1+\gamma)((k+1)+\beta k)}E_{k-1}^{\beta^{2}}
\end{align*}
where we have used the equation from the previous step in the parenthesis. Proceeding by complete induction this implies
$$E_{k+1}\leq \tilde{C}^{1+\beta+\ldots+\beta^{k}} 2^{(1+\gamma)((k+1)+\beta k+\ldots + \beta^{k})}E_{0}^{\beta^{k+1}}$$

Noticing that the bound
$$ \sum_{j=1}^{k+1} j\beta^{k+1-j}\leq  \beta^{k}\sum_{j=0}^{\infty} \frac{j}{\beta^j} = \frac{\beta^{k+1}}{(\beta-1)^2}$$
holds for $\beta>1$ we conclude that 
\begin{equation}
    E_{k+1}\leq  \tilde{C}^{\frac{\beta^{k+1}-1}{\beta-1}}2^{\frac{\beta^{k+1}}{(\beta-1)^2}}E_{0}^{\beta^{k+1}}.
\end{equation}
Changing the constant by a larger constant $C''=C''(\tilde{C},\beta, 2)>1$ this can be rewritten as
$$E_{k+1}\leq   (C''\epsilon)^{\beta^{k+1}}\leq\epsilon^{(1-\eta)\beta^{k+1}}$$
taking $\epsilon=\epsilon(\beta, C, C')$ such that $C''<\epsilon^{-\eta}$. We will choose $\eta>0$ small enough so that $(1-\eta)\beta>1$.
In particular, this shows that under the smallness hypothesis  
$E_{0}=\epsilon\leq t_{\star}^{-\gamma/(2-\gamma)}$, implies
\[E_{k+1}\leq\epsilon^{(1-\eta)\beta^{k+1}}\leq\epsilon^{(1-\eta)\beta}\leq\epsilon= E_0\leq t_{\star}^{-\gamma/(2-\gamma)}\]
which proves the first part of the inductive hypothesis; the second one follows immediately this and equation  \eqref{final:Ek:recurrence}.  This provides the induction step and concludes the proof that $E_{k}$ tends to zero as $k$ grows. As a consequence, recalling the definition of $E_k$,
$$E_{k}=\int_{M}\theta_{k}(x,t_{k}) \ d\textrm{vol}_g(x), $$
we have shown that
\begin{equation}\label{conclusion:theta}
    \norm{\theta(\cdot,t_{\star})}_{L^\infty}=\displaystyle \sup_{x\in M}\theta(x,t_{\star}) \leq \frac{C'}{t_\star^{\gamma}}.
\end{equation}
Summarizing for any smooth positive initial profile $\theta_{0}$ satisfying  the smallness assumption
\[\|\theta_{0}\|_{L^1}:=\epsilon(\beta, C, C')\leq t_{\star}^{-\gamma/(2-\gamma)}\]
and $\|\theta_{0}\|_{L^\infty}=1$, the $L^\infty$ decay  in \eqref{conclusion:theta} provides a contradiction for times $t_\star=\displaystyle\max\{C',\epsilon^{-1}\}^2$ for $\epsilon=\epsilon(\beta, C,C')$ as classical solutions to \eqref{eq:active:scalar2} would make $\max\theta (\cdot,t)=1$ for all $t>0$.\\

Notice that this choices are quite arbitrary as, for example, a contradiction is also reached if we choose $t_{\star}=1$, $C'=\frac{1}{2}\|\theta_0\|_{\infty}$ and $\epsilon$ small enough. This proves our remark after the statement in the introduction.

\begin{appendix}
\section{Local well-posedness in Sobolev spaces }\label{appendix}

In this section, for the reader's convenience, we will sketch the local well-posedness theory in the manifold setting. The reason why we only sketch it is that in the case of manifolds there are difficulties inherent to the setting as we will emphasize below. The main reason is that the usual commutators have to be adapted as the objects involved, such as the gradient, can not be understood {\em componentwise} globally. Nevertheless, the proof we sketch is complete in the case of $n$-dimensional tori and we also hint on the changes one should perform to obtain local well-posedness in large Sobolev spaces.\\

In particular, we will sketch the proof that the system \eqref{eq:active:scalar} in $H^{s}(M)$ for $s>1+\frac{n}{2}$ is wellposed. To prove the analogous local well-posedness result for the generalized active scalar equation \eqref{eq:active:scalar2} in the range  $-1<\alpha<0$ one can mimick the same arguments we will provide below. However, the range $0<\alpha<1$ poses extra difficulties due to the more singular velocity field $u$.\\

The authors in \cite{CCCGW} rewrote the nonlinear term in the form of a commutator to explore the extra cancellation providing the well-posedness of the system, which allowed them to deal with the singular scenario in the particular cases of $\mathbb{R}^n$ or $\mathbb{T}^n$. Sharper results in terms of the Sobolev exponent where obtained later using finer analysis of the commutator estimate (see \cite{LiComm}). Whether those delicate commutator estimates can be adapted to the compact manifold setting is however out of the scope of this paper and might be of independent interest.\\

After this brief introduction let us provide the proof of local in time existence of solution in Sobolev spaces of system \eqref{eq:active:scalar}. The strategy is classical and relies on a deriving adequate {\em a priori} energy estimates, see \cite{TAY2} for well-posedness results of general hyperbolic systems in Riemannian manifolds. The regularized system is given by \begin{equation}
\left\{
\begin{array}{lcl}\label{eq:appendix:1}
\partial_{t}\theta_{\epsilon}+\mathcal{J}_{\epsilon}(\mathcal{J}_{\epsilon}u_{\epsilon}\cdot\nabla_{g} \mathcal{J}_{\epsilon}\theta_{\epsilon})=0, \quad \mbox{on }M\times(0,\infty)\\
u_{\epsilon}= \nabla_{g}\Lambda_{g}^{-1}\theta_\e \\
\theta_{\epsilon}(x,0)=\mathcal{J}_{\epsilon}\theta_{0}(x), \quad \mbox{in } M
\end{array}
\right.
\end{equation}
where  $\mathcal{J}_{\epsilon}=e^{\epsilon\Delta_{g}}$ is the Friedrichs mollifier.\\

The first step is to obtain local in time estimates of $\theta_{\epsilon}\in H^{s}(M)$ which are independent of $\epsilon$ and imply the existence of an uniform (in $\epsilon$) time of existence. Finally, by compactness pass to the limit. \\

We begin  estimating the $L^2$ norm.  Taking the inner $L^2$ product and integrating by parts we have that
\begin{align}\label{estimate:L2}
\frac{1}{2}\frac{d}{dt} \norm{\theta_{\e}}^2_{L^{2}} &=-\frac{1}{2} \langle \mathcal{J}_{\epsilon} \textrm{div}_{g} u_{\epsilon},\abs{\mathcal{J}_{\epsilon}\theta_{\e}}^2 \rangle + \frac{1}{2} \langle [\textrm{div}_{g}, \Je ]u_\e,\abs{\mathcal{J}_{\epsilon}\theta_{\e}}^2 \rangle \\ \nonumber
&\lesssim  \norm{\textrm{div}_{g} u_\epsilon}_{H^{s}} \norm{\theta_\e}^2_{L^2} \lesssim  \norm{u_\epsilon}_{H^{s+1}} \norm{\theta_\e}^2_{L^2}
\end{align}
where we have used the fact that the principal symbol of $[\textrm{div}_{g}, \Je ]\in OPS^{0}_{1,0}$ (cf. \cite{TAY})  and the Sobolev embedding $H^{s}(M)\hookrightarrow L^\infty(M)$, $s>\frac{n}{2}$.\\

To derive the higher-order Sobolev $H^s$ norms, recall the notation $\Lambda_{g}=(-\Delta_{g})^{\frac{1}{2}}$ and $\norm{f}_{\dot{H}^s}=\norm{\Lambda^s_{g}f}_{L^2}$. Then we have that

\begin{align*}
\frac{1}{2}\frac{d}{dt} \norm{\Lambda^{s}_{g}\theta_{\e}}^2_{L^{2}} &=-\langle \Lambda^{s}_{g}\Je (\mathcal{J}_{\epsilon} u_{\epsilon}\cdot \nabla_{g} \Je \theta_\e),\Lambda^s_{g}\theta_{\e} \rangle = I_{1}+I_{2}
\end{align*}
where
$$I_{1}=\langle [\Lambda^{s}_{g},\Je] (\mathcal{J}_{\epsilon} u_{\epsilon}\cdot \nabla_{g} \Je \theta_\e),\Lambda^s_{g}\theta_{\e} \rangle$$
and
$$I_2= \langle \Lambda^{s}_{g}(\mathcal{J}_{\epsilon} u_{\epsilon}\cdot \nabla_{g} \Je \theta_\e),\Je \Lambda^s_{g }\theta_{\e} \rangle.$$
Since the principal symbol of the commutator $[\Lambda^{s}, \Je ]\in OPS^{s-1}_{1,0}$ (cf.  \cite{TAY, TAY2}), and the fact that $H^{s-1}$ is an algebra for $s>\frac{n}{2}+1$, we infer that
\begin{align*}
   \abs{I_{1}}& \lesssim  \norm{\mathcal{J}_{\epsilon} u_{\epsilon}\cdot \nabla_{g} \Je \theta_\e}_{H^{s-1}}\norm{\Lambda^s_{g}\theta_{\e}}_{L^2}\\
&\lesssim \norm{\mathcal{J}_{\epsilon} u_{\epsilon}}_{H^{s-1}}\norm{\nabla_{g} \Je \theta_\e}_{H^{s-1}}\norm{\Lambda^s_{g}\theta_{\e}}_{L^2} \\
&\lesssim \norm{u_{\epsilon}}_{H^{s-1}}\norm{\theta_{\e}}^2_{\dot{H}^s}.
\end{align*}
Rearranging the latter term we have that
\[\begin{aligned}
I_2&=\langle \Lambda^{s}_{g}(\mathcal{J}_{\epsilon} u_{\epsilon}\cdot \nabla_g \Je \theta_\e), [\Je, \Lambda^s_{g}]\theta_{\e} \rangle+\langle \Lambda^{s}_{g}(\mathcal{J}_{\epsilon} u_{\epsilon}\cdot \nabla_g \Je \theta_\e), \Lambda^s_{g}\Je\theta_{\e} \rangle\\
&=\langle \Lambda^{s-1}_{g}(\mathcal{J}_{\epsilon} u_{\epsilon}\cdot \nabla_g \Je \theta_\e), \Lambda_g[\Je, \Lambda^s_{g}]\theta_{\e} \rangle+\langle \Lambda^{s}_{g}(\mathcal{J}_{\epsilon} u_{\epsilon}\cdot \nabla_g \Je \theta_\e), \Lambda^s_{g}\Je\theta_{\e} \rangle\\
&=I_{21}+I_{22}.
\end{aligned}\]

In a similar way as for the term $I_1$ we have that
$$ \abs{I_{21}}\lesssim \norm{u_{\epsilon}}_{H^{s-1}}\norm{\theta_{\e}}^2_{\dot{H}^s}.$$
To bound $I_{22}$ we notice that
\begin{align*}
I_{22} &=\langle \mathcal{J}_{\epsilon} u_{\epsilon}\cdot  \nabla_{g}\Lambda^{s}_{g} \Je \theta_\e, \Lambda^s_{g}\Je\theta_{\e} \rangle + \langle [\Lambda^{s},\mathcal{J}_{\epsilon} u_{\epsilon}\cdot \nabla_g] \Je \theta_\e, \Lambda^s_{g}\Je\theta_{\e} \rangle \\
&= \frac{1}{2}\langle \mathcal{J}_{\epsilon} u_\epsilon, \nabla_{g}(\abs{\Lambda^{s}_{g}\Je\theta_{\e}})^2 \rangle + \langle [\Lambda^{s},\mathcal{J}_{\epsilon} u_{\epsilon}\cdot \nabla_g] \Je \theta_\e, \Lambda^s_{g}\Je\theta_{\e} \rangle \\
&\lesssim \norm{u_\epsilon}_{H^{s+1}} \norm{\Lambda^{s}_{g}\theta_\e}^2_{L^2} + \langle [\Lambda^{s},\mathcal{J}_{\epsilon} u_{\epsilon}\cdot \nabla_g] \Je \theta_\e, \Lambda^s_{g}\Je\theta_{\e} \rangle 
\end{align*}
where we have used integration by parts and the same bounds as in the $L^2$ estimate  \eqref{estimate:L2}. To conclude, we need to bound the last term (which is the most singular). The following algebraic identity\footnote{This identity, although true for the $n$-dimensional tori does not make sense in general as it would involve a fractional Laplace-Beltrami operator of a vector field, which is not defined here. The same applies later to the application of Kato-Ponce commutator estimate. Nevertheless, if one is willing to work with even $s$ using local coordinate charts it can be handled through a rather tedious but elementary argument.}
\begin{equation}
    [\Lambda^s,G\cdot\nabla]F=[\Lambda_{g}^s,G]\nabla_{g} F+G[\Lambda^s_{g},\nabla_{g}]F
\end{equation}
with $F=J_{\epsilon}\theta_{\epsilon}$ and $G=J_{\epsilon}u_{\epsilon}$ implies that it equals
\begin{equation}\label{singular:term:final}
\langle [\Lambda^{s}_{g},\mathcal{J}_{\epsilon} u_{\epsilon}\cdot] \nabla_{g}\Je \theta_\e, \Lambda^s_{g}\Je\theta_{\e} \rangle +  \langle \mathcal{J}_{\epsilon} u_{\epsilon}\cdot [\Lambda^{s}_{g}, \nabla_{g}]\Je \theta_\e, \Lambda^s_{g}\Je\theta_{\e} \rangle.
\end{equation}
The latter term is a lower order term (since $[\Lambda^{s}_{g}, \nabla_{g} ]\in OPS^{s-1}_{1,0}$) and can be bounded easily by 
$$ \abs{\langle \mathcal{J}_{\epsilon} u_{\epsilon}\cdot [\Lambda^{s}, \nabla_{g}]\Je \theta_\e, \Lambda^s_{g}\Je\theta_{\e} \rangle} \lesssim \norm{u_\epsilon}_{H^{s}} \norm{\Lambda^{s}_{g}\theta_\e}^2_{L^2} $$
for $s>\frac{n}{2}.$ To bound the first term in \eqref{singular:term:final} we make use of the following Kato-Ponce commutator estimate (cf. \S 3.6 in \cite{TAY})
\begin{equation}\label{KP:estimate}
    \norm{\Lambda^{s}_{g}(fg)-f\Lambda^{s}_{g}g}_{L^2}\lesssim \norm{f}_{C^1}   \norm{g}_{H^{s-1}}  +\norm{f}_{H^s}\norm{g}_{L^\infty}.
\end{equation}
Kato-Ponce inequality together with the Sobolev embeddings $H^s(M)\hookrightarrow C^1(M)$, $s>1+\frac{n}{2}$, and $H^s(M)\hookrightarrow L^{\infty}(M)$, $s>\frac{n}{2}$, respectively, yields
\[\|[\Lambda^s_{g},f]g\|_{L^2}\lesssim \|f\|_{H^s}\|g\|_{H^{s-1}}.\]
Using this for the choices $f=\mathcal{J}_{\epsilon} u_{\epsilon}$ and $g=\nabla_{g}\Je \theta_\e$ we have that 
\begin{align*}
\abs{ \langle [\Lambda^{s}_{g},\mathcal{J}_{\epsilon} u_{\epsilon}\cdot] \nabla_{g}\Je \theta_\e, \Lambda^s_{g}\Je\theta_{\e} \rangle}\lesssim \norm{u_\e}_{H^s}\norm{\theta_\e}^2_{H^s}.
\end{align*}
Therefore, collecting the $L^2$ estimates and $\dot{H}^s$ estimates we have shown that 
\begin{equation} \label{final:bound}
    \frac{1}{2}\frac{d}{dt} \norm{\theta_\epsilon}^2_{H^s} \lesssim \norm{u_\epsilon}_{H^s}\norm{\theta_\epsilon}^2_{H^s}
\end{equation}
for any $s>\frac{n}{2}+1$. Moreover, since $u_\e= \nabla_{g}\Lambda^{-1}_{g}\theta_\e$ and the principal symbol of the operator is in $OPS^{0}_{1,0}$ we have that $\norm{u_\epsilon}_{H^s}\lesssim \norm{\theta_\epsilon}_{H^s}$ and hence
\begin{equation} \label{final:bound:2}
    \frac{1}{2}\frac{d}{dt} \norm{\theta_\epsilon}^2_{H^s} \lesssim \norm{\theta_\epsilon}^3_{H^s}.
\end{equation}
which gives the explicit time interval existence of the solution
\begin{equation*}
    \norm{\theta_\epsilon}_{H^s} \leq \textcolor{black}{ \frac{\norm{\theta_{0,\e}}_{H^s}}{\left(1-Ct\norm{\theta_{0,\e}}_{H^s}\right)}}
\end{equation*}
which immediately provides a uniform bound for $t\in[0,c\|\theta_0\|_{H^s}^{-1}]$ for some small $c>0$ and $\epsilon$ small enough.\\

As we mentioned in the introduction, to conclude the proof one can apply standard techniques (cf. \S 17 of \cite{TAY2}) to obtain a solution $\theta(t)\in C\left( [0,T), H^s(M)\right)$ for initial data $\theta_0\in H^s(M)$ with  $s>1+\frac{n}{2}$. The same estimates above can be used to provide the uniqueness of local solutions.

\end{appendix}

\vspace{.2in}
\noindent{\bf{Acknowledgments.}}
The authors would like to thank A. C\'ordoba for drawing their attention to this problem several years ago. They are also grateful to the referee for pointing out several misprints and helping to improve the presentation's clarity.\\

D. Alonso-Or\'an is supported by the Spanish MINECO through Juan de la Cierva fellowship FJC2020-046032-I. The major part of this work was done when D. Alonso-Or\'an was supported by the Alexander von Humboldt Foundation. A. D. Mart\'inez is supported by a Fields Ontario Postdoctoral Fellowship financed by NSERC Discovery Grant 311685 and NSERC Grant RGPIN-2018-06487.


\begin{thebibliography}{10}
\bibitem{AOCM}
Alonso-Or\'an, D.; C\'ordoba, A.; Mart\'inez, A. D., {\em Continuity of weak solutions of the critical quasigeostrophic equations on the sphere},  Advances in Mathematics
Volume 328, 13 April 2018, pp. 264-299.

\bibitem{AOCM2}
Alonso-Or\'an, D.; C\'ordoba, A.; Mart\'inez, A. D., {\em Global well--posedness of critical surface quasigeostrophic equation on the sphere},  Advances in Mathematics
Volume 328, 13 April 2018, pp. 248-263.


\bibitem{Bakeretal}
Baker, G. R.; Li, X;  Morlet, A. C., {\em Analytic structure of two 1D-transport equations with non-local fluxes.} Physica D: Nonlinear Phenomena, 91(4): pp. 349–375, 1996.
\bibitem{BC}
Balodis, P.; C\'ordoba, A., {\em An inequality for Riesz transforms implying blow-up for some nonlinear and non-local transport equations}, Adv. of Math. Vol. 214 (10) (2007), pp. 1-39.


\bibitem{CorCas}
Castro, A.; C\'ordoba, D., {\em Infinite energy solutions of the surface quasi-geostrophic equation}, Adv. in Math. 225 (2010), pp. 1820-1829.
\bibitem{Chae}
Chae, D., {\em On the transport equations with singular/regular non-local velocities}, SIAM J. Math. Anal. 26, 2 (2014), pp. 1017-1029.
\bibitem{CCCGW}
Chae, D. ; Constantin, P.  C\'ordoba, D., Gancedo, F. and  Wu, J. {\em Generalized surface quasi‐geostrophic equations with singular velocities }, Communications on Pure and Applied Mathematics 65, Issue 8 (2012), pp. 1037-1066.
\bibitem{Ch}
Chavel, I., {\em Eigenvalues in Riemannian geometry}, Academic Press, 1984.

\bibitem{CY}
Chae, D.; Constatin, P.; Wu, J., {\em Inviscid models generalizing the two-dimensional Euler and the surface quasi-geostrophic equations}, ARMA 202, 1 (2011), pp. 35-62.
\bibitem{CLM}
Constantin, P.; Lax, P. D.; Majda, A., {\em A simple one‐dimensional model for the three‐dimensional vorticity equation},  Comm. Pure Appl. Math. 38 (6), pp. 715-724.
\bibitem{CMT}
Constantin, P. ;  Majda, A. and Tabak, E. {\em Formation of strong fronts in the 2-D quasigeostrophic thermal active scalar}. Nonlinearity 7.6 (1994), pp. 1495–1533.

\bibitem{CC} 
C\'ordoba, A.; C\'ordoba, D., {\em A Maximum Principle Applied to Quasi-Geostrophic Equations}, Commun. Math. Phys. 249 (2004), pp. 511-528.
\bibitem{CCF}
C\'ordoba, A.; C\'ordoba, D.; Fontelos, M. A., {\em Formation of singularities for a transport equation with non-local velocity}, Ann. of Math., 162 (2005), pp. 1377-1389.

\bibitem{CM}
C\'ordoba, A.; Mart\'inez, A. D., {\em A pointwise inequality for fractional laplacians}, Adv. of Math., Vol. 280 (2015), pp. 79-85.
\bibitem{DG}
De Giorgi, E., {\em Sulla differenziabilit\`a e l'analiticit\`a delle estremali degli integrali multipli regolari}, Mem. Accad. Sci. Torino Cl. Sci. Fis. Mat. Nat. (3) 3 (1957), pp. 25-43.

\bibitem{Dong1}
Dong, H., {\em Well-posedness for a transport equation with non-local velocitiy}, J. Funct. Anal. 225, 11 (2008), pp. 3070-3097.
\bibitem{DongLi1}
Dong, H.; Li, D., {\em Finite time singularties for a class of generalized surface quasi-geostrophic equations}, Proc. Amer. Math. Soc. 136, 11 (2008), pp. 2555-2563.
\bibitem{DongLi2}
Dong, H.; Li, D., {\em On a one-dimensional $\alpha$-patch model with non-local drift and fractional dissipation}, Trans. Amer. Math. Soc. 366, 4 (2014), pp. 2041-2061.


\bibitem{Ho}
H\"ormander, L., {\em  The spectral function of an elliptic operato},  Acta Math. 121 (1968), pp. 193-218.
\bibitem{HPGS95}
 Held, I.M.;  Pierrehumbert, R. T.; Garner,  S.T.; Swanson, K. L.,
{\em Surface quasi-geostrophic dynamics.} J. Fluid Mech., 282: pp. 1–20, 1995.
\bibitem{Ju}
Ju, N.,  {\em Dissipative quasi-geostrophic equation: local well-posedness, global regularity and similarity solutions}. Indiana Univ. Math. J. Vol. 56, No. 1 (2007), pp. 187-206.
\bibitem{Kiselev1}
Kiselev, A. {\em Regularity and blow up for active scalars. }, Math. Model. Nat. Phenom. 5 (2010), no. 4, 225–255.
\bibitem{KP}
Kato, T.; Ponce, G., {\em Commutator estimates and the Euler and Navier–Stokes equations}, Comm. Pure Appl. Math. 41 (1988), pp. 891-907.

\bibitem{KNV}
Kiselev, A.; Nazarov, F.; Volberg, A., {\em Global well-posedness for the critical 2D dissipative quasigeostrophic equation}, Invent. Math. 167 (2007), pp. 445-453.

\bibitem{Jiu-Zhang}
Jiu, Q.; Wanwan Z. {\em Formation of singularities for multi-dimensional transport equations with nonlocal velocity}, arXiv:2111.0144 (2021).
\bibitem{LiComm}
Li, D. {\em On Kato-Ponce and fractional Leibniz. }, Rev. Mat. Iberoam. 35 (2019), no. 1, 23–100.

\bibitem{LiRodrigo1}
Li, D.; Rodrigo. J., {\em Blow-up of solutions for a 1D transport equation with non-local velocity and supercritical dissipation}, Adv. Math. 217, 6 (2008), pp. 2563-2568.
\bibitem{LiRodrigo2}
Li, D.; Rodrigo. J., {\em Blow up for the generalized quasi-geostrophic equation with supercritical dissipation}, Comm. Math. Phys. 286, (2009), pp. 111-124.

\bibitem{LiRodrigo3}
Li, D.; Rodrigo. J., {\em Remarks on a non-local transport}, Advances in Mathematics Volume 374, 18 November 2020, pp. 107-345

\bibitem{Morlet}
Morlet, A. C., {\em Further properties of a continuum of model equations with globally defined flux,} Journal of mathematical
analysis and applications, 221(1): pp. 132–160, 1998.
\bibitem{Ped}
Pedlosky, P., 
{\em Geophysical Fluid Dynamics.} Springer Verlag, 1982.
\bibitem{R}
Resnick, S. G. {\em Dynamical problems in non-linear advective partial differential equations}, PhD thesis (Chicago University, 1995).


\bibitem{SV}
Silvestre, L.; Vicol, V., {\em 
On a transport equation with non-local drift}, Trans. Amer. Math. Soc. 368 (2016), no. 9, pp. 6159-6188.
\bibitem{So}
Sogge, Ch. D., {\em The Hangzhou lectures on Eigenfunctions of the Laplacian}, Princeton University Press, 2014.

\bibitem{S1}
Stein, E. M., {\em Singular integrals and Differentiability Properties of Functions}, Princeton University Press, 1970.

\bibitem{TAY}   
Taylor, M., {\em Pseudodifferential Operators and Nonlinear PDE}, Birkhäuser, Boston, 1991.
\bibitem{TAY2}
Taylor, M., {\em Partial Differential Equations}, vols. 1-3, Springer-Verlag, New York, 1996 (2nd ed., 2011).

\end{thebibliography}
\end{document}